\documentclass[10pt]{amsart}
\RequirePackage[colorlinks,citecolor=blue,urlcolor=blue,linkcolor=blue]{hyperref}
\hypersetup{
colorlinks = true,
citecolor=blue,
urlcolor=blue,
linkcolor=blue,
pdfpagemode = UseNone
}
\usepackage{graphicx,xspace,colortbl,url}
\usepackage{amsmath,amsthm,amsfonts,setspace}
\usepackage{color}
\usepackage{fancybox}
\usepackage{epsfig}
\usepackage{subfig}
\usepackage{pdfsync}
\usepackage{enumerate}
\usepackage{geometry,soul}
    \def\qed{\hfill$\sqcap\kern-8.0pt\hbox{$\sqcup$}$\\}
    \def\beq{\begin{eqnarray}}
    \def\eeq{\end{eqnarray}}
    \def\beqq{\begin{eqnarray*}}
    \def\eeqq{\end{eqnarray*}}
    \def\zz{{\mathbb Z}}
    
    \def\re{\textnormal {Re}}
    \def\im{\textnormal {Im}}
    
    \def\e{{\mathbb E}}
    \def\r{{\mathbb R}}

    \def\d{{\textnormal d}}
    \def\i{{\textnormal i}}

    \def\vv{{\textnormal v}}

\newtheorem{theorem}{Theorem}[section]
\newtheorem{lemma}[theorem]{Lemma}
\newtheorem{proposition}[theorem]{Proposition}
\newtheorem{corollary}[theorem]{Corollary}
\theoremstyle{definition}
\newtheorem{definition}{Definition}[section]

\newtheorem{assumption}{Assumption}[section]
\newtheorem{remark}[theorem]{Remark}

\numberwithin{equation}{section}
\newcommand{\la}{\lambda}
\newcommand{\eps}{\varepsilon}

  \newcommand{\hide}[1]{}

\def\vv#1{{\boldsymbol #1}}

\newcommand{\RR}{\r}

\newcommand{\ZZ}{\zz}

\newcommand{\mJ}{\vv {J}}
\newcommand{\mX}{\vv {X}}
\newcommand{\mY}{\vv {Y}}

\newcommand{\E}{\e}
\def\V{{\rm Var}}
\def\Var{{\rm Var}}

\def\topp#1{^{(#1)}}
   
\newcommand{\TT}{\mathbb{T}}
\newcommand{\T}{\mathbf{T}}
\newcommand{\C}{  \color{blue} \mathsf c  \color{black}}
\newcommand{\A}{  \color{blue} \mathsf a \color{black}}
\newcommand{\B}{ \color{blue} \mathsf b \color{black}}

\title{On the continuous dual Hahn process}
\author{W{\l}odek Bryc}
\address
{
W\l odzimierz Bryc\\
Department of Mathematical Sciences\\
University of Cincinnati\\
2815 Commons Way\\
Cincinnati, OH, 45221-0025, USA.
}
\email{wlodek.bryc@gmail.com}


 \usepackage{framed,pgf}
\newcounter{oldeq}
\newcounter{usesofarxiv}
 \newcommand{\arxiv}[1]{
\setcounter{oldeq}{\value{equation}}
 \addtocounter{usesofarxiv}{1}
 \setcounter{equation}{0}
\def\theoldeq{\theequation}
\def\theequation{x-\arabic{usesofarxiv}.\arabic{equation}}
\def\theequation{\arabic{section}.\arabic{usesofarxiv}.\arabic{equation}}
\def\theequation{\thesection.\arabic{usesofarxiv}.\arabic{equation}}
  \colorlet{shadecolor}{gray!10}
{\footnotesize
\begin{shaded}#1
\end{shaded}
   \setcounter{equation}{\value{oldeq}}
\numberwithin{equation}{section}
}\color{black}}
\begin{document}

\maketitle

\begin{abstract}
We extend  the continuous dual Hahn process $(\TT_t)$ of Corwin and Knizel  from a finite time interval   to the entire real line by taking  a limit of a closely related Markov process $(\T_t)$.
We also  characterize    processes $(\T_{t})$ by  conditional means and  variances under bidirectional conditioning, and we prove that continuous dual Hahn polynomials are orthogonal martingale polynomials for  both processes.

\end{abstract}

\arxiv{This is an expanded version of the paper with additional material.}
\section{Introduction}
In this paper we are interested in a family of Markov transition probabilities on the real line which are constructed from the orthogonality measures of the continuous dual Hahn polynomials.
Together with the appropriate marginal laws that arise from a point mass as the initial law, these transition probabilities define a class of continuous time Markov processes $(\T_t)$  which  appeared in the construction of quadratic harnesses in \cite[Section 3]{Bryc:2009}. Together  with an appropriate family of $\sigma$-finite positive measures as the entrance laws, see \eqref{ps-invar},  these transition probabilities appeared in the description of the multipoint Laplace transform for  stationary measures of the open KPZ equation in \cite[Theorem 1.4(5)]{CorwinKnizel2021}.
 Following \cite{CorwinKnizel2021},  we shall use the name the continuous dual Hahn process, and we will use their suggestive notation  $(\TT_t)$.
Our goal is to extend the time domain   of the process $(\TT_t)$ from a finite interval   described in \cite[(1.10)]{CorwinKnizel2021}, to the real line. %
The need for an extension  of the time domain arose in \cite[Theorem 1.3]{Bryc-Kuznetsov-Wang-Wesolowski-2021}, although for the purposes of that paper the extension to $t\in [0,\infty)$  would suffice.
We accomplish our goal by analyzing the  Markov process   $(\T_t)$ as one of its parameters diverges to $\infty$.

The actual process $(\TT_t)$ constructed here differs slightly from  the continuous dual Hahn  process in \cite{CorwinKnizel2021}: the process that appears in Refs \cite{Bryc-Kuznetsov-Wang-Wesolowski-2021,CorwinKnizel2021}    corresponds to $(4\TT_{s/2})$. On the other hand, $(\T_t)$   as constructed in this note, is  a direct extension of the family of Markov processes from \cite{Bryc:2009}   to  a half-line as the time domain. %
We will obtain the entrance laws for the process $(\TT_t)$  by taking  a limit   of    the appropriately scaled marginal laws for the process $(\T_t)$.

Our approach to the construction, which relies on verification of the Chapman-Kolmogorov equations, is somewhat different than in Refs. \cite{Bryc:2009,CorwinKnizel2021}, which used   explicit formulas for the orthogonality measures of the continuous dual Hahn polynomials. In the presence of atoms, such  explicit formulas lead to proliferation of cases, which we avoid by relying on   properties of the orthogonal martingale polynomials for   $(\T_t)$. In particular, as in \cite[Section 3.2]{Bryc-Wesolowski-08}, we deduce   the Chapman-Kolmogorov equations from the algebraic relations between two families of orthogonal polynomials.

 The paper is organized as follows. In Section \ref{Sec:CDH-poly} we recall the  definition of the continuous dual Hahn polynomials and discuss the probability measures  which make them orthogonal. In Section \ref{Sec:Trans} we use these measures to construct the family of transition probabilities and marginal laws for Markov process $(\T_t)$. Our main result, Theorem \ref{Thm:Chapman},  establishes the Chapman-Kolmogorov equations. In Section \ref{Sec:CDH-inf} we
introduce the $\sigma$-finite entrance laws that define  process $(\TT_t)$ for all $t\in(-\infty,\infty)$.
 In Section \ref{Sec:QH} we characterize the Markov process $(\T_t)$ by the formulas for the conditional mean and the conditional variance.

\section{Continuous dual Hahn polynomials}\label{Sec:CDH-poly}
\subsection{Favard's Theorem}
We first recall a version of  Favard's theorem in the form that encompasses  in one statement orthogonality with respect to both finitely supported and   infinitely supported measures.
This form of Favard's theorem is "well known" to the experts and it is implicit in many proofs,  in particular in the argument presented  in \cite[Section 2.5]{Ismail-book}. The explicit reference  (with a proof) is
  \cite[Theorem A.1]{Bryc-Wesolowski-08}.
\begin{theorem}[Favard's Theorem]\label{Thm-Favard} Let $\alpha_n,\beta_n$ be real, $n\geq 0$. Consider monic
polynomials
 $\{p_n\}$ defined by the recurrence %
  \begin{equation}\label{FVD-rec-m}
x {p}_n(x)=  p_{n+1}(x)+\alpha_n  p_n(x)+\beta_n  p_{n-1}(x),\; n\geq 0,
\end{equation}
with  the initial conditions $p_0(x)=1$,  $p_{-1}(x)=0$. Then the following two conditions are equivalent:
 \begin{enumerate}[(i)]
 \item For all $n\geq 1$,
\begin{equation}
  \label{Favard_condition}
  \prod_{j=1}^n \beta_{j}\geq 0 .
\end{equation}
\item
There exists a (not necessarily unique) probability measure $\nu$ with  all moments such that   for all $m,n\geq 0$,
\begin{equation}\label{FVD-ortho}
\int    p_n(x)   p_m(x)\nu(dx)=\delta_{m,n}\prod_{j=1}^{n} \beta_{j}.
\end{equation}
\end{enumerate}
Furthermore,   suppose that \eqref{Favard_condition} holds. Then either $\beta_n>0$  for all $n\geq 1$, and then measure $\nu$ has infinite support, or there is a positive integer $n\geq 1$ such that  $\beta_n=0$. In the latter case, denote by  $N$  the first positive integer such that $\beta_{N}=0$.
Then condition \eqref{Favard_condition}  contains no further  restrictions on the values $\beta_n$ for $n>N$ and the orthogonality measure
$\nu(dx)$ is a (unique) discrete probability measure supported on the finite set of
$N\geq 1$ real and distinct zeros of the polynomial
$ p_{N}(x)$.
\end{theorem}
\subsection{The three step recurrence for the continuous dual Hahn polynomials}
The continuous dual Hahn polynomials  are monic polynomials which depend
on three parameters. These parameters are
traditionally denoted by $a,b,c$, but  to avoid confusion with the parameters $\A, \B, \C$ for the Markov process $(\T_t)$, we will denote them by $\alpha,\beta,\gamma$.
We always assume that parameter $\alpha$ is real, and that parameters $\beta,\gamma$ are either   both real or form a complex conjugate pair.
Then sequences
\begin{equation}
  \label{AnCn}
  A_n=(n+\alpha+\beta)(n+\alpha+\gamma),  \quad C_n=n(n-1+\beta+\gamma), \quad n=0,1,\dots
\end{equation}
are real.
The continuous dual Hahn polynomials, see \cite[(1.3.5)]{koekoek1998askey},  are  monic  polynomials $\{p_n(x|\alpha,\beta,\gamma)\}$ in real variable $x$,
defined by the
three step recurrence relation
\begin{equation}\label{Recursion}
x p_n(x|\alpha,\beta,\gamma)= p_{n+1}(x|\alpha,\beta,\gamma)+(A_n+C_n-\alpha^2)p_n(x|\alpha,\beta,\gamma)+A_{n-1}C_n p_{n-1}(x|\alpha,\beta,\gamma),
\end{equation}
$n=0,1,\dots$, with the usual initialization $p_{-1}(x|\alpha,\beta,\gamma)=0$, $p_0(x|\alpha,\beta,\gamma)=1$; then \eqref{Recursion} gives
\begin{equation}
  \label{p1} p_1(x|\alpha,\beta,\gamma)=x-\alpha \beta-\alpha \gamma-\beta \gamma.
\end{equation}

By comparing recursion \eqref{Recursion} with \cite[(1.3.4)]{koekoek1998askey}, we get
\begin{equation}\label{p2F3}
  p_n(x|\alpha,\beta,\gamma)=(-1)^n (\alpha+\beta,\alpha+\gamma)_n\;{_3}F_2(-n,\alpha-\sqrt{-x},\alpha+\sqrt{-x}; \alpha+\beta,\alpha+\gamma;1),
\end{equation}
where
\begin{equation}\label{3F2}
  {_3}F_2 (a_1,a_2,a_3;b_1,b_2;z)=\sum_{k=0}^\infty \frac{(a_1,a_2,a_3)_k}{(b_1,b_2)_k}\frac{z^k}{k!}
\end{equation}
denotes the generalized hypergeometric function.
Here and throughout the paper we use the following notation for the products of Gamma functions and  the   Pochhammer symbols:
$$\Gamma(a,b,\dots,c)= \Gamma(a)\Gamma(b)\dots\Gamma(c),\quad (a)_n=a(a+1)\dots(a+n-1), \quad (a,b,\dots,c)_n=(a)_n(b)_n\dots (c)_n.$$

With the above restrictions on the parameters, polynomials $\{p_n(x|\alpha,\beta,\gamma)\}$  are well defined and real valued, but they do not have to be orthogonal.  Favard's theorem %
  allows us to recognize for which choices of the parameters %
polynomials
$\{p_n(x|\alpha,\beta,\gamma)\}$ are orthogonal.
If parameters $\alpha,\beta,\gamma$ are such that \eqref{Favard_condition} holds with $\beta_n=A_{n-1}C_n$,
then polynomials
$p_n(x|\alpha,\beta,\gamma)$ are orthogonal in the following sense:  there exists a %
probability measure $ \nu(\d x|\alpha,\beta,\gamma)$ such that
\begin{equation}
  \label{nu-def} \int_\r p_n(x|\alpha,\beta,\gamma)p_m(x|\alpha,\beta,\gamma)\nu(\d x|\alpha,\beta,\gamma)=0
\end{equation} for $m\ne n$. From \eqref{FVD-ortho} it is clear that in the case of a measure with $N$ atoms, the integral \eqref{nu-def} is zero also for $m=n\geq N$.

Measures $\nu(\d x|\alpha,\beta,\gamma)$ play a prominent role in our construction, as we will define the transition probabilities and the marginal distributions for the Markov process $(\T_t)$ by   specifying the parameters
$\alpha,\beta,\gamma$. We will refer to $ \nu(\d x|\alpha,\beta,\gamma)$ as the orthogonality measure for the polynomials $\{p_n(x|\alpha,\beta,\gamma)\}$.

For the constructions, we  need to know  that the orthogonality measure $\nu(\d x|\alpha,\beta,\gamma)$ is unique, i.e., that it is determined by moments. This fact should be known, but we did not find   a published reference.
 So for completeness we adapt an argument   from an unpublished preprint
  \cite[Proposition 3.1]{Bryc-Matysiak-Szwarc-Wesolowski-06}, who  considered a larger   family of polynomials in a different parametrization.

\begin{lemma}
  \label{Lem:unique} Orthogonality measures for the polynomials defined by \eqref{Recursion} are determined uniquely by moments.
\end{lemma}
\begin{proof}
  Since finitely supported measures are determined uniquely by moments, we only need to consider the case when $A_kC_{k+1}>0$ for all $k\geq 0$. In particular,  we assume that $A_0C_1=(\alpha+\beta)(\alpha+\gamma)(\beta+\gamma)>0$.
  We will use a criterion that involves the  numerator polynomials $q_n(x)$, which solve  recursion \eqref{Recursion} with the initial conditions $q_0(x)=0$, $q_1(x)=1$ and $n\geq 1$ (see e.g. \cite[Section 2.3]{Ismail-book} or \cite[Section 2.1]{Akhiezer}).
   Let
\begin{equation}\label{p orthonormal}
\widetilde{p}_n(x)=\frac{1}{\sqrt{ A_0A_1\dots
A_{n-1}C_1C_2\dots C_n}}\;p_n(x|\alpha,\beta,\gamma) \end{equation} and
$$\widetilde{q}_n(x)=\frac{1}{\sqrt{ A_0A_1\dots
A_{n-1}C_1C_2\dots C_n}}\;\;q_n(x)
$$ be the corresponding normalized polynomials.
  By a theorem of Hamburger, \cite[page 84]{Akhiezer}, the moment problem is determined
uniquely, if and only if at some point $x_0\in\RR$ we have
\begin{equation}\label{infinity}
\sum_n|\widetilde{p}_n(x_0)|^2+\sum_n|\widetilde{q}_n(x_0)|^2=\infty.
\end{equation}
We shall verify that this condition holds with $x_0=-\alpha^2$.

By \eqref{p2F3},
$$p_n(-\alpha^2|\alpha,\beta,\gamma)=(-1)^n (\alpha+\beta,\alpha+\gamma)_n.$$
Therefore, noting that
\begin{equation}
   A_0A_1\dots
A_{n-1}C_1C_2\dots C_n=n!(\alpha+\beta,\alpha+\gamma,\beta+\gamma)_n,
\end{equation}
we have
$$|\widetilde{p}_n(-\alpha^2)|^2 =\frac{(\alpha+\beta,\alpha+\gamma)_n}{n!(\beta+\gamma)_n} \sim \frac{\Gamma(\alpha+\beta+n,\alpha+\gamma+n)}{n! \Gamma(\beta+\gamma+n)}\sim \frac{1}{n^{1-2\alpha}},$$
where we write $a_n\sim b_n$ if $\lim_{n\to\infty} a_n/b_n \in(0,\infty)$.  Thus \eqref{infinity} holds if $\alpha\geq 0$.

To verify that \eqref{infinity} holds for $\alpha\leq 0$, we analyze the numerator polynomials.  With $x=-\alpha^2$, recursion \eqref{Recursion} simplifies to
$$
q_{n+1}(-\alpha^2)+A_n q_n(-\alpha^2)=-C_n\left(q_n(-\alpha^2)+A_{n-1}q_{n-1}(-\alpha^2)\right), \quad n\geq 1.
$$
So with $q_0=0,q_1=1$ we have
$$
q_{n+1}(-\alpha^2)+A_n q_n(-\alpha^2)=(-1)^{n}\prod_{k=1}^n C_k,\quad n\geq 0.
$$
It is easy to check that the solution of this recursion (with $q_0=0,q_1=1$)
is
$$q_n(-\alpha^2)=(-1)^{n-1} \left(\prod_{k=1}^{n-1}A_k \right) \sum_{m=0}^{n-1}\prod_{k=1}^{m}\tfrac{C_k}{A_k}, \quad n\geq 0.
$$
After normalization, we get
\begin{equation}
 |\widetilde{q}_n(-\alpha^2)|^2 = \frac{1}{A_0^2}\left(\prod_{k=0}^{n-1}\frac{A_k}{C_{k+1}}\right)
 \left(\sum_{m=0}^{n-1} \prod_{k=1}^{m}\tfrac{C_k}{A_k}\right)^2.
\end{equation}
To verify \eqref{infinity} we now use the fact that
\begin{equation}
  \label{elementary estimates}
\frac{C_{n}}{A_n}=1-\frac{1+2\alpha}{n}+O(1/n^2),\;\frac{A_n}{C_{n+1}}=1
-\frac{1-2\alpha}{n}+O(1/n^2).\end{equation}
Thus
\begin{multline}
   |\widetilde{q}_n(-\alpha^2)|^2 \sim \prod_{k=0}^{n-1}\left(1-\frac{1-2\alpha}{k}\right)
 \; \left(\sum_{m=0}^{n-1} \prod_{k=1}^{m}\left(1-\frac{1+2\alpha}{k}\right)\right)^2
\\
 \sim \exp\left( -\sum_{k=0}^{n-1}\tfrac{1-2\alpha}{k} \right)\;\left(\sum_{m=0}^{n-1} \exp (-\sum_{k=1}^{m}\tfrac{1+2\alpha}{k})\right)^2
 \\
 \sim \frac{1}{n^{1-2\alpha}}\left(\sum_{m=0}^{n-1} \frac1{m^{1+2\alpha}}\right)^2
\sim \frac{1}{n^{1-2\alpha}} n^{-4\alpha}= \frac{1}{n^{1+2\alpha}}.
\end{multline}
Thus \eqref{infinity} holds also if $\alpha\leq 0$, completing the proof.
\end{proof}

 \section{Transition probabilities and the marginal laws}\label{Sec:Trans}

  Our first goal is to define a family of Markov kernels
  $\{\mathfrak p_{s,t}(x, \d y ): s<t\}$ which will serve as the transition probabilities. These kernels will be chosen from the orthogonality measures of the continuous dual Hahn polynomials and depend on  a real parameter $\C$.
  A subtle point in the construction is that due to restriction \eqref{Favard_condition} these orthogonality measures cannot be defined  for all real $x$, and the excluded set of $x$'s depends on the value of $s$. This leads to a rarely considered case of Markov processes with  a "time-dependent" state space, as    in \cite[Sections 9, 10]{dynkin1978sufficient}.

 For $s\in\r$, we  introduce a family of  sets
  \begin{equation}\label{Es}
   E_s= \begin{cases}
      [-(\C-s)^2,\infty) & s\leq \C, \\
     \{-(\C-s+N)^2: N= 0,1,\dots,\; N<s-\C \} \cup  [0,\infty) & s>\C.
   \end{cases}
\end{equation}
The orthogonality measures of the continuous dual Hahn polynomials allow us to define measures $\mathfrak p_{s,t}(x, \d y )$   at  spatial locations $x\in E_s$.
 We will extend artificially our definition to  $x\not\in E_s$, but as in  \cite{dynkin1978sufficient},   the resulting Markov process will really be defined  on the product of the sets $E_s$.
Thus sets $E_s$   play an important role in the construction of the Markov family, and will appear in several statements below.

 For $-\infty<s<t<\infty$, we define the family of Markov kernels by inserting times $s,t$ and location $x$ into the parameters
 of  the orthogonality measure introduced in formula \eqref{nu-def}.
  Let
 \begin{equation}\label{P_st-3par}
\mathfrak p_{s,t}(x, \d y ):=  \begin{cases}
  \nu(\d y|  \C-t,  t-s-\i\sqrt{x},t-s+\i\sqrt{x}) & x\in E_s,\; x\geq 0, \\
\nu(\d y|  \C-t,  t-s-\sqrt{-x},t-s+\sqrt{-x}) & x\in E_s,\;  x<0,\\
\delta_{-(c-t)^2}(\d y) & x\not \in E_s.
\end{cases}
 \end{equation}

 We need to verify that these probability measures are well defined and that their supports are contained in the corresponding sets $E_t$.

\begin{proposition}\label{Prop-trans}
Probability measures \eqref{P_st-3par} are well defined for all $-\infty<s<t<\infty$.
Furthermore,
\begin{equation}
  \label{P(Et)=1}
  \mathfrak p_{s,t}(x,E_t)=1.
\end{equation}
\end{proposition}
\begin{proof} It is clear that the conclusion holds if $x\not\in E_s$, so we only need to consider $x\in E_s$.
   To verify that the probability measure is well defined,  we analyze the factors in product \eqref{Favard_condition} with $ \beta_{n}=A_{n-1} C_{n}$. With parameters as specified in \eqref{P_st-3par}, from \eqref{AnCn}  we get
    \begin{equation}
      \label{F2} A_n C_{n+1}=(n+1) (n+2 (t-s)) \left((\C+n-s)^2+x\right).
    \end{equation}
First, consider the boundary case $x=-(\C-s)^2$. Then $A_0C_1=0$, so by Theorem \ref{Thm-Favard}  measure $\mathfrak p_{s,t}(x, \d y )$  is  concentrated at the root $\alpha\beta+\alpha\gamma+\beta\gamma=-(\C-t)^2$ of  polynomial \eqref{p1}.
 We get $\mathfrak p_{s,t}(x, \d y )=\delta_{-(\C-t)^2}(\d y)$. In particular, we have $\mathfrak p_{s,t}(x, E_t )=1$.

Next, consider the non-boundary cases $x\in E_s$  with $x\geq 0$. Then products \eqref{F2} are strictly positive, so  measure  $\mathfrak p_{s,t}(x, \d y )$ is well defined and has infinite support (except for the already considered boundary case of $x=0$, $\C=s$).  In \eqref{P_st-3par}, we have  measure $\nu(\d y|\C-t,t-s+\i\sqrt{x}, t-s-\i\sqrt{x})$ with the complex-conjugate pair of parameters with positive real part $t-s$. From \cite[Section 1.3]{koekoek1998askey} we see that
     $\mathfrak p_{s,t}(x, \d y )$ has absolutely continuous component with a density supported on $(0,\infty)$, and if
 $t>\C$ then in addition to the absolutely continuous component, measure $\mathfrak p_{s,t}(x, \d y )$ has also a discrete component with atoms at points $y_k=-(\C -t+k)^2$ for $k=0,1,\dots$ such that $t>\C+k$. Thus  $\mathfrak p_{s,t}(x, E_t )=1$.

Finally, we consider the non-boundary cases with $x\in E_s$ such that  $-(\C-s)^2<x<0$.
\begin{enumerate}[(A)]
  \item If $s<\C$, then products \eqref{F2} are positive, as $A_n C_{n+1}\geq A_0C_1>0$. So measure  $\mathfrak p_{s,t}(x, \d y )$ is well defined and has infinite support. It remains to verify that $\mathfrak p_{s,t}(x, E_t )=1$.

  Since $x<0$, we have  $x=-v^2$ for some $0<v<\C-s$ and  in \eqref{P_st-3par}, we have  measure  $\nu(\d y|\C-t,t-s+v, t-s-v)$ with three real parameters.
  Since $t-s+v>0$  and
  $t-s-v>t-\C$,  at least two of the parameters of measure $\nu(\d y|\C-t,t-s+v, t-s-v)$  are  positive: the second positive parameter is either $\C-t>0$ or  $t-s-v>0$.  From \cite[Section 1.3]{koekoek1998askey} we see that
     $\mathfrak p_{s,t}(x, \d v )$ has  absolutely continuous component with a density supported on $(0,\infty)$  and  with atoms
at points $y_k=-(\C -t+k)^2$, if $\C-t<0$, or at points $\tilde y_k=-(t-s-v+k)^2$ if $t-s-v<0$, and then  $\C-t> 0$. It is clear that points $y_k$ are in $E_t$. On the other hand, if $t-s-v<0$, then $\tilde y_k\geq \tilde y_0=-(v+s-t)^2> -(\C-t)^2$ so $\tilde y_k\in E_t$, as we have $\C> t$ in this case.

\item If $s>\C$ and $x\in E_s$, then   $x=-(\C-s+N)^2$ for some $N$ such that $s>\C+N$.
 We see that \eqref{F2} factors as
$$A_n C_{n+1}=(n+1) (n+2 (t-s))  \left(2\C+n+N-2s \right) \left(n-N\right).$$
 Since $2\C+n+N-2s\leq 2(\C+N-s)<0$ for $n\leq N$,   the last two factors are both negative for $n=0,1,\dots,N-1$, and $A_N C_{N+1}=0$. So the products \eqref{Favard_condition} are positive and then 0.
     By Theorem \ref{Thm-Favard}, measure   $\mathfrak p_{s,t}(x,\d y)$ is atomic with $N+1$ atoms at the roots of its $(N+1)$-th orthogonal polynomial $ p_{N+1}(y|\alpha,\beta,\gamma)$ as written in  \eqref{Q-poly}, which in view of \eqref{3F2Q}  factors as
  $$Q_{N+1}(y;x,t,s)=(\C-t-\sqrt{-y},\C-t+\sqrt{-y})_{N+1}=
  \prod_{k=0}^N((\C-t+k)^2+y).$$
 The roots of this polynomial are $-(\C-t)^2, -(\C-t+1)^2, \dots, -(\C-t+N)^2$ and they lie in $E_t$.
\end{enumerate}
\end{proof}

\begin{remark}\label{Rem:det}
   From the proof of Proposition \ref{Prop-trans} we note that if  $x=-(\C-s)^2$, then measure $\mathfrak p_{s,t}(x,\d y)=\delta_{-(\C-t)^2}(\d y)$ is degenerate. As a consequence, a Markov process with transition probabilities $\mathfrak p_{s,t}(x,\d y)$  which is at location $x\not\in E_s$ at time $s$, will follow the parabola $t\mapsto - (\C-t)^2$  at the boundary of sets $E_t$ for $t>s$.
\end{remark}
\begin{remark}
  \label{Rem:Wojtek} Wojciech Matysiak pointed out to us  that for $t>s>\C$ the orthogonality measure $\nu(\d y|  \C-t,  t-s-\sqrt{-x},t-s+\sqrt{-x})$  is well defined also for all $x\in(-\min_{k=0,1,\dots}(\C-s+k)^2,0)$ which are not in $E_s$. The form of the orthogonality measure for this case seems to be unknown, and we expect that in addition to the  expressions listed in \cite[Section 1.3]{koekoek1998askey} there is an additional component that allows for convergence to $\delta_x$ as $t\searrow s$. However, under Assumption \ref{A1}, such $x$'s are not within the support of the marginal laws, so   we restrict our construction to $x\in E_s$ and set  $\mathfrak p_{s,t}(x, \d y )=\delta_{-(\C-t)}(\d y)$ for $x\not\in E_s$.
\end{remark}

\subsection*{Marginal laws}
The marginal laws for the Markov process are also defined using the orthogonality measures of the continuous  dual Hahn polynomials. We use two additional parameters $\A,\B$  which satisfy the following.
\begin{assumption}\label{A1}  We assume one of the following:
\begin{enumerate}[(a)]
\item $\A,\B,\C$ are real parameters such that $\A+\C> 0$, with $\B\geq \A$, or
\item    $\C$ is real, and $\A,\B$ are complex conjugates with $\im(\A)\ne 0$.
\end{enumerate}
\end{assumption}
(Since the expressions below are symmetric in parameters $\A,\B$,   condition $\B\geq \A$ in Assumption \ref{A1}(a)  is  just a convenient labeling convention.)

We use the orthogonality measure from  formula \eqref{nu-def} to define a family of marginal  {probability} laws   as follows:
  \begin{equation}
   \label{p-uni} \mathfrak p_t(\d x\mid \A,\B,\C) := \nu(\d x|\C-t, \A+t,\B+t ).
 \end{equation}
 We need to verify that the definition is correct and  to state  explicit formulas that will be needed in  Section \ref{Sec:CDH-inf}.
\begin{proposition}\label{P-Koekoek} Under Assumption \ref{A1}, probability measures \eqref{p-uni}
 are well defined for all $t\geq -(\A+\B)/2$.
 Furthermore,  $\mathfrak p_t(E_t\mid \A,\B,\C)=1$, and the explicit formulas for the measures are as follows.

  If $t=-(\A+\B)/2$, then $\mathfrak p_t(\d x\mid \A,\B,\C)$ is a degenerate measure $\delta _{-(\A-\B)^2/4}(\d x)$.

 If  $t>-(\A+\B)/2$, then   $\mathfrak p_t(\d x\mid \A,\B,\C)=\mathfrak p_t\topp c(\d x\mid \A,\B,\C)+\mathfrak p_t\topp d(\d x\mid \A,\B,\C)$  is the sum of the continuous and  discrete components.
 The continuous component is supported on $(0,\infty)$ and is given by
 \begin{multline}\label{p-c}
   \mathfrak p_t\topp c(\d x\mid \A,\B,\C)
   \\ =
   \frac{1}{4 \pi\Gamma (\A+\C, \B+\C,\A+\B+2 t)} \cdot \frac{|\Gamma(\A+t+\i \sqrt{x},\B+t+\i  \sqrt{x},\C-t+\i \sqrt{x} )|^2}{ \sqrt{x}|\Gamma(2 \i \sqrt{x})|^2} 1_{x>0} \d x.
 \end{multline}
 The discrete component is either zero, or it has a finite number of atoms in $(-\infty, 0)$. The discrete component is non-zero  in the following two cases.
 \begin{enumerate}[(a)]
    \item If $\A$ is real and $t+\A<0$ then
   \begin{equation}
     \label{pd-m} \mathfrak p_t\topp d(\d x\mid \A,\B,\C)=\sum_{\{k\geq 0: \; \A+t+k<0\}} m_t(k) \delta_{-(\A+t+k)^2}(\d x)
   \end{equation}
      with
   \begin{equation*}
      m_k(t) =
      \frac{ \Gamma (-\A+\C-2 t)}{\Gamma (-2 (\A+t)) }\cdot \frac{(\A+k+t) (\A+\C)_k (2 (\A+t))_k }{k! (\A+t) (\A-\C+2 t+1)_k}\cdot \frac{\Gamma (\B-\A)(\A+\B+2 t)_k}{\Gamma (\B+\C) (\A-\B+1)_k}(-1)^k.
   \end{equation*}
   \item If $t>\C$ then
    \begin{equation}
      \label{pd-M}\mathfrak p_t\topp d(\d x\mid \A,\B,\C)=\sum_{\{k\geq 0:\; \C-t+k<0\}} M_k(t) \delta_{-(\C-t+k)^2}(\d x)
    \end{equation}
   with
      \begin{multline*}
      M_k(t) = \frac{(\C+k-t)}{k! (\C-t)}\cdot \frac{\Gamma (\A-\C+2 t) }{\Gamma (2 (t-\C)) }
      \cdot
     \frac{ (\A+\C ,2 (\C-t))_k}{ (-\A+\C-2 t+1)_k }\cdot \frac{\Gamma (\B-\C+2 t)}{\Gamma (\A+\B+2 t)}\cdot \frac{(\B+\C)_k}{(-\B+\C-2 t+1)_k}(-1)^k.
   \end{multline*}

 \end{enumerate}
\end{proposition}
(To facilitate taking the limit as $\B\to\infty$ in a later argument,  factors with $\B$ are separated at the end of the formulas.)
\arxiv{Note that if $\A$ is real then Assumption \ref{A1} implies that $-\A+\C-2t=(\A+\C)-2(\A+t)>0$ in case (a) and $\B-\C+2t\geq \A-\C+2t=(\A+\C)+2(t-\C)>0$ in case (b), so the Gamma functions  that appear  in the formulas are well defined.}
\begin{proof}
If $t=-(\A+\B)/2$, then in the product \eqref{Favard_condition}  we get $A_0C_1=0$, giving a degenerate measure at the root $x_0=\alpha\beta+\alpha\gamma+\beta\gamma$ of   polynomial \eqref{p1}. We get
$\mathfrak p_{t}(\d x\mid \A,\B,\C)=\delta_{x_0}$  with  $x_0=- (\A-\B)^2/4$. We now check that $x_0\in E_t$.   If $\A=\bar \B$, then $x_0\geq0$, so $x_0\in E_t$. If the parameters are real, then $\B\geq \A>-\C$, so
$0\leq ({\B-\A})/{2}=-\A-t<\C-t$. Therefore,  $E_t=[-(\C-t)^2,\infty)$ and
$$x_0=-(\tfrac {\B-\A}{2})^2>-(\C-t)^2$$ is in $E_t$.

  For $t>-(\A+\B)/2$,  we have  a measure $\nu(\d x|\C-t, \A+t,\B+t )$ that either has to two complex-conjugate parameters $\A+t,\B+t$ with positive real part (as $t>-(\A+\B)/2=-\re(\A)=-\re(\B)$), or with three real parameters of which at least two are positive. Indeed, we have $\B+t\geq (\A+\B)/2+t>0$ and if  $\A+t\leq 0$ then
$\C-t\geq\C+\A>0$.
So the distribution can be read out from  \cite[Section 1.3]{koekoek1998askey}.
\begin{enumerate}[(a)]
  \item If either $-\A\leq t\leq \C$ or $\im(\A)\ne 0$ and $t\leq \C$ then  the distribution has density \eqref{p-c} supported on $(0,\infty)$ and has no atoms.
 \item  If
$\A+t< 0$,  then in addition to the absolutely continuous component \eqref{p-c}, there are  atoms \eqref{pd-m}, which are in $E_t$,  as from $0<-\A-t<\C-t$ we get $E_t=[-(\C-t)^2,\infty)$, see \eqref{Es}.
  \item If $t>\C$ then in addition to the absolutely continuous component \eqref{p-c}, there are
 atoms  \eqref{pd-M}  that  are in $E_t$.
 \end{enumerate}   So in all three  cases, $\mathfrak p_{t}(E_t\mid \A,\B,\C)=1$.
\end{proof}

We now verify that
the family of probability measures
$\mathfrak p_t(\d x\mid \A,\B,\C)$
together with the family of Markov kernels $\mathfrak p_{s,t}(x,\d y)$ satisfy the Chapman-Kolmogorov equations.

\begin{theorem}\label{Thm:Chapman}
Suppose that parameters $\A,\B,\C$ satisfy    Assumption \ref{A1}.
Let $U$ be a Borel subset of $\r$.
\begin{enumerate} [(i)]
\item
For $-\infty<s<t<u<\infty$ and $x\in\r$, we have
\begin{equation}
  \label{Chap-Kolm-1}
  \int_\r \mathfrak p_{s,t}(x,\d y)\mathfrak p_{t,u}(y,U)= \mathfrak p_{s,u}(x,U).
\end{equation}

\item For $-(\A+\B)/2\leq s<t<\infty$,
\begin{equation}
  \label{Chap-Kolm-2}
  \int_\r \mathfrak p_s(\d x\mid \A,\B,\C) \mathfrak p_{s,t}(x,U) = \mathfrak p_t(U\mid \A,\B,\C).
\end{equation}

\end{enumerate}
\end{theorem}
The following definition summarizes the above.
\begin{definition} \label{Def-T}
With parameters $\A,\B,\C$ which satisfy Assumption \ref{A1} and $\tau=-(\A+\B)/2$, we denote by
  $(\T_s)_{s\geq \tau}$   a Markov process  marginal laws \eqref{p-uni} and with transition probabilities  \eqref{P_st-3par}.
  The process starts at time $\tau=-(\A+\B)/2$ at the deterministic location $x_\tau=-(\A-\B)^2/4$.
\end{definition}

\begin{remark}
  The continuous dual Hahn polynomials are a limiting case of the four-parameter family of Wilson polynomials \cite{Wilson:1980}.
It would be interesting to see how the four-parameter orthogonality measures of Wilson polynomials could be used to construct Markov processes. Two special cases are known: Ref.  \cite[Section 2]{Bryc:2009} considered the absolutely continuous case, and Ref. \cite{Bryc-Matysiak-11} considered a purely atomic case.
\end{remark}

 The rest of this section is devoted to the proof of Theorem \ref{Thm:Chapman}. In the proof we use the following two families of polynomials. Denote
  \begin{equation}\label{p-mart}
  p_n(x;s)= p_n(x\mid \C-s, \A+s,\B+s),
  \end{equation}
    \begin{equation}
      \label{Q-poly}
      Q_n(y;x,t,s)=p_n(y\mid\C-t,  t-s-\sqrt{-x},t-s+\sqrt{-x}).
    \end{equation}
    (With the convention that $\pm\sqrt{-x}$ is $\pm \i\sqrt{x}$ when $x>0$, compare  \eqref{P_st-3par}.)

    Polynomials $p_n(x;s)$ are of course the monic orthogonal polynomials for the univariate laws  $\mathfrak p_s(\d x\mid \A,\B,\C)$.  If $x\in E_s$ and $s<t$,  then polynomials $Q_n(y;x,t,s)$ are the orthogonal polynomials for the measures $\mathfrak p_{s,t}(x,dy)$. However, recursion \eqref{Recursion} defines polynomials $Q_n(y;x,t,s)$   for all $x,s,t\in\r$, and we will need  this more general setting for some of the arguments.

   The key step in the proof is the following algebraic fact about the
   connection coefficients between these two families of polynomials.
\begin{lemma}
  \label{Lem-connect} If $\A,\B,\C$ satisfy Assumption \ref{A1}, then
  there exist functions $\{b_{n,k}(x,s):1\leq k\leq n\}$ (in fact, polynomials in $s,x$) which do not depend on $t$ such that $b_{n,n}(x,s)=1$, and for all $x,y\in\r$ we have
\begin{equation}
  \label{Q2p}
     Q_n(y;x,t,s)=\sum_{k=0}^n b_{n,k}(x,s)p_k(y;t), \quad n=1,2,\dots
\end{equation}
\end{lemma}
We remark that by linear independence, any set of monic polynomials can be expressed as a unique linear combination of the polynomials $p_k(y;t)$, $k=0,1,\dots$. The main point  of Lemma \ref{Lem-connect} is that the coefficients of the linear combination \eqref{Q2p} do not depend on variable $t$.
\begin{proof}%
From \eqref{p2F3}, we see that each polynomial
$$p_n(y;t)=(-1)^n(\A+\C,\B+\C)_n \, _3F_2\left(-n,\C-t-\sqrt{-y},\C-t+\sqrt{-y};\A+\C,\B+\C;1\right)$$
is a linear combination of  the linearly independent monic polynomials
\begin{equation}
  \label{Pochhamers}
  (\C-t-\sqrt{-y},\C-t+\sqrt{-y})_k=\prod_{j=0}^{k-1} ((\C-t+j)^2+y), \quad k=0,1,\dots,n
\end{equation}
 in variable $y$. From  \eqref{3F2} we see that the  coefficients  of the linear combination,
 $$
 (-1)^n\frac{(\A+\C,\B+\C)_n(-n)_k}{k!(\A+\C,\B+\C)_k}, \quad k=0,1,\dots, n,
 $$
  do not depend on $t$.  By Assumption \ref{A1}, either $\B+\C\geq \A+\C>0$ or $\im(\A)\ne 0$, so $(\A+\C,\B+\C)_n\ne 0$ and hence the coefficients are non-zero.
 This means that polynomials \eqref{Pochhamers} can be written as linear combinations of polynomials $p_0(y;t),p_1(y;t),\dots,p_n(y;t)$ with the coefficients that do not depend on $t$.
Since
\begin{multline}\label{3F2Q}
   Q_n(y;x,t,s)=\;
   _3F_2\left(-n,\C-t-\sqrt{-y},\C-t+\sqrt{-y};\C-s-\sqrt{-x},\C-s+\sqrt{-x};1\right)
   \\
   \cdot (-1)^n(\C-s-\sqrt{-x},\C-s+\sqrt{-x})_n,
\end{multline}
 from \eqref{3F2} we see that $Q_n(y;x,t,s)$  is a linear combination of the monic polynomials \eqref{Pochhamers}, with
   the  coefficients
   $$
   (-1)^n\frac{(-n)_k(\C-s-\sqrt{-x},\C-s+\sqrt{-x})_n}{k!(\C-s-\sqrt{-x},\C-s+\sqrt{-x})_k },\quad
   k=0,1,\dots,n
   $$
   that depend only on $x,s$, but not on $t$. Combining these two observations together,  we get
   \eqref{Q2p}.
   The fact that $b_{n,n}(x,s)=1$ is just a consequence of the fact that both sets of  polynomials are monic.

\end{proof}

We  note that $Q_n(x;x,s,s)=0$ for $n\geq 1$.  To see this, as in \cite{Bryc-Wesolowski-08},  we  use
the three step recurrence relation \eqref{Recursion} to verify that  $Q_1(x;x,s,s)=0$,  $Q_2(x;x,s,s)=0$, and then automatically \eqref{Recursion}  implies that $Q_n(x;x,s,s)=0$ for all $n\geq 3$.

From  \eqref{Q2p}, applied to  each term in $Q_n(y;x,t,s)= Q_n(y;x,t,s)-Q_n(x;x,s,s)$, after canceling the first term with $p_0(y;t)=1$, we get
\begin{equation}
  \label{Q2p+}
     Q_n(y;x,t,s)=\sum_{k=1}^n b_{n,k}(x,s)(p_k(y;t)-p_k(x;s)), \quad n=1,2,\dots
\end{equation}
Formula \eqref{Q2p+} immediately implies that  $\{p_n(x;s)\}$ are in fact orthogonal martingale polynomials for the Markov process $(\T_s)$. This implication is known, see \cite[Proposition 3.6]{Bryc-Wesolowski-08} but we include proof for completeness.
  \begin{proposition}
    \label{L-mart}  If $x\in E_s$, then
    \begin{equation}\label{proj-mart}
        \int_\r p_n(y;t)\mathfrak p_{s,t}(x,\d y)=p_n(x;s).
    \end{equation}
  \end{proposition}
\begin{proof} The proof is by induction on $n$.
Trivially, \eqref{proj-mart} holds for $n=0$, as  $p_0(x;s)=1$. For the induction step, suppose that the martingale property \eqref{proj-mart} holds for   polynomials
  $p_k(x;t)$  with $k\leq n-1$, where  $n\geq 1$. Since  for $x\in E_s$
      polynomials $Q_n(y;x,t,s)$ and $ Q_0(y;x,t,s)\equiv 1$  are orthogonal,
 from \eqref{Q2p+}
  we get
   \begin{multline*}
   0=\int_\r Q_{n}(y;x,t,s)\mathfrak p_{s,t}(x,\d y)
    \\
    =\sum_{k=1}^{n-1} b_{n,k}(x,s)\int_\r (p_k(y;t)-p_k(x;s))\mathfrak p_{s,t}(x,\d y)+b_{n,n}(x,s)\int_\r (p_n(y;t)-p_n(x;s))\mathfrak p_{s,t}(x,\d y)
   \\ \stackrel{\eqref{proj-mart}}{=} 0+\int_\r (p_n(y;t)-p_n(x;s))\mathfrak p_{s,t}(x,\d y).
 \end{multline*}
where we used   the induction assumption for all the terms with $k\leq n-1$  and    we used $b_{n,n}(x,s)=1$   in the last term. Thus
 $\int_\r (p_n(y;t)-p_n(x;s))\mathfrak p_{s,t}(x,\d y)=0$, which ends the proof by induction.

\end{proof}
We remark that  for $x\not\in\E_s$, polynomials $Q_n(y;x,t,s)$ do not have to be orthogonal,   so \eqref{proj-mart} may fail if $x\not\in E_s$.

\begin{proof}[Proof of Theorem \ref{Thm:Chapman}]
We note that \eqref{Chap-Kolm-1} holds by default  if $x\not\in E_s$,  as then the process is deterministic, see Remark \ref{Rem:det}.

To prove that \eqref{Chap-Kolm-1} holds for $x\in E_s$, following \cite[page 1242]{Bryc-Wesolowski-08} we introduce
  an auxiliary probability  measure
  $$\mu(U)=\int_\r \mathfrak p_{s,t}(x,\d y)\mathfrak p_{t,u}(y,U)=\int_{E_t}\mathfrak p_{s,t}(x,\d y)\mathfrak p_{t,u}(y,U),$$
   as $\mathfrak p_{s,t}(x,E_t)=1$.
Then  for $n=1,\dots$, we have
  \begin{equation}
    \label{mu*}
    \int_\r Q_{n}(z;x,u,s) \mu(\d z) =0.
  \end{equation}
  Indeed,
 \begin{multline*}
    \int_\r Q_{n}(z;x,u,s) \mu(\d z)= \int_\r Q_{n}(z;x,u,s) \int_\r \mathfrak p_{s,t}(x,\d y) \mathfrak p_{t,u}(y,\d z)
    \\= \int_\r \mathfrak p_{s,t}(x,\d y) \int_\r Q_{n}(z;x,u,s)\mathfrak p_{t,u}(y,\d z)
    \\\stackrel{\eqref{Q2p+}}{=}\sum_{k=1}^n b_{n,k}(x,s) \int_\r \mathfrak p_{s,t}(x,\d y) \int_\r(p_k(z;u)-p_k(x;s))\mathfrak p_{t,u}(y,\d z)
     \\\stackrel{\eqref{P(Et)=1}}{=}\sum_{k=1}^n b_{n,k}(x,s) \int_{E_t} \mathfrak p_{s,t}(x,\d y) \int_\r(p_k(z;u)-p_k(x;s))\mathfrak p_{t,u}(y,\d z)
    \\ \stackrel{\eqref{proj-mart}}{=} \sum_{k=1}^n b_{n,k}(x,s) \int_\r \mathfrak p_{s,t}(x,\d y)  (p_k(y;t)-p_k(x;s))  \stackrel{\eqref{proj-mart}}{=}0,
 \end{multline*}
 where we used martingale property first at $y\in E_t$, and then again at $x\in E_s$.

  It is well known, see e.g. \cite[Exercise 2.5]{Ismail-book} that condition \eqref{mu*} implies that the moments of $\mu(\d z)$   are the same as the moments of the orthogonality measure $\mathfrak p_{s,u}(x,\d z)$ for the  polynomials $Q_{n}(z;x,u,s)$, $n\geq0$.  By Lemma \ref{Lem:unique}, we get   $\mu(U)= \mathfrak p_{s,u}(x,U)$ for all Borel sets $U$, proving  \eqref{Chap-Kolm-1}. %

  To prove that \eqref{Chap-Kolm-2} holds, we follow a similar plan. Recycling the same letter, we introduce another
 auxiliary  probability  measure
  $$\mu(U)=  \int_\r \mathfrak p_s(\d x\mid \A,\B,\C) \mathfrak p_{s,t}(x,U)= \int_{E_s} \mathfrak p_s(\d x\mid \A,\B,\C) \mathfrak p_{s,t}(x,U).$$
 For $n=1,2,\dots$ , we use   martingale property \eqref{proj-mart} on $E_s$  to compute
 \begin{multline*}
    \int_\r p_n (y;t)\mu(\d y)=\int_\r p_n (y;t)\int_\r \mathfrak p_s(\d x\mid \A,\B,\C) \mathfrak p_{s,t}(x,\d y)
    \\  =\int_{E_s} \mathfrak p_s(\d x\mid \A,\B,\C)\int_\r  p_n (y;t)\mathfrak p_{s,t}(x,\d y) \stackrel{\eqref{proj-mart}}{=}
    \int_\r p_n (x;s)\mathfrak p_s(\d x\mid \A,\B,\C)=0,
 \end{multline*}
where in the last step we used orthogonality of polynomials $p_n(x;s)$ and $p_0(x;s)\equiv1$ with respect to $\mathfrak p_s(\d x\mid \A,\B,\C)$. Since polynomials
 $\{p_n(y;s)\}$ are orthogonal with respect to probability measure $\mathfrak p_t(\d y\mid \A,\B,\C)$,
 by the uniqueness of the moment problem, $\mu(\d y)=\mathfrak p_t(\d y\mid \A,\B,\C)$, proving \eqref{Chap-Kolm-2}.

 (In both proofs, the interchange of the order of integrals is allowed, as the measures have all moments.)
\end{proof}

\section{A family of $\sigma$-finite  entrance laws}\label{Sec:CDH-inf}
In this section we consider  real parameters $\A,\B,\C$, under   Assumption \ref{A1}(a). (Parameter $\B$ will appear only in the proof.)

  For $-\infty<t<\infty$, we  introduce a family of  $\sigma$-finite measures
$\mathfrak p_t(\d x)$,  which depend on parameters $\A,\C$.  These measures were motivated by Ref. \cite[Definition 7.8]{CorwinKnizel2021}.
\begin{definition} For real $t,\A,\C$ with $\A+\C>0$, consider a family of positive $\sigma$-finite measures
  \begin{multline}
    \label{their-nu}
   \mathfrak p_{t}(\d x)= \mathfrak p_{t}\topp c(\d x)+ \mathfrak p_{t}\topp d(\d x) :=
    \frac{1}{
    4\pi} \frac{|\Gamma(t+\A+\i \sqrt{x},\C-t+\i \sqrt{x} )|^2}{\sqrt{x}|\Gamma(2\i \sqrt{x})|^2}1_{x>0}\d x
    \\+
    \sum_{\{j:\; j+\A+t<0\}} m_j(t) \delta_{-(\A+j+t)^2}(\d x)+
    \sum_{\{k:\; \C-t+k<0\}}  M_k(t) \delta_{-(\C-t+k+t)^2}(\d x)
  \end{multline}
  with discrete masses given by
    \begin{equation}
    \label{their-atoms}
    m_j(t)
    =\frac{(\A+j+t)}{j! (\A+t)}\cdot \frac{ \Gamma(\A+\C,\C-\A-2 t)}{\Gamma (-2 (\A+t)) }\cdot\frac{ (\A+\C,2 (\A+t))_j }{ (\A-\C+2 t+1)_j},\; j\in \ZZ\cap[0,-\A-t),
  \end{equation}
    \begin{equation}
    \label{their-other-atoms}
    M_k(t)
    = \frac{(\C+k-t)}{k! (\C-t)}\cdot\frac{\Gamma (\A-\C+2 t,\A+\C) }{\Gamma (2 (t-\C)) }
      \cdot
     \frac{ (\A+\C ,2 (\C-t))_k}{ (-\A+\C-2 t+1)_k }, \; k\in\ZZ\cap[0,t-\C).
  \end{equation}
  \end{definition}

Our goal is to show that measures $\mathfrak p_t(\d x)$ are the
 entrance laws (\cite[Sect. 10]{dynkin1978sufficient}) for the family of transition probabilities $\mathfrak p_{s,t}(x,\d y)$ defined by \eqref{P_st-3par}, i.e., that for all Borel sets $U$ we have %
 \begin{equation}
  \label{ps-invar}
   \int_\r \mathfrak p_{s}(\d x)\mathfrak p_{s,t}(x,U) = \mathfrak p_{t}(U).
\end{equation}
The following extends \cite[Lemma 7.11]{CorwinKnizel2021}    to a larger time domain.
\begin{theorem}\label{Prop-inv}
Suppose $\A,\C$  are real and   $\A+\C>0$. Then for $-\infty<s<t<\infty$, and all  Borel sets $U$,
the entrance law formula \eqref{ps-invar} holds.

\end{theorem}

 Theorem \ref{Prop-inv} follows from the explicit formulas for $\mathfrak p_t(\d x\mid \A,\B,\C)$ in Proposition \ref{P-Koekoek}. The idea of proof is that for fixed $t>-(\A+\B)/2$, we have
$$\frac{\Gamma(\A+\C)\Gamma (\B+\C) \Gamma (\A+\B+2 t)}{\Gamma (\B+t)^2} \mathfrak p_t(\d x\mid \A,\B,\C)
\to \mathfrak p_t(\d x)$$
 as $\B\to \infty$, and that this limit preserves   equation \eqref{Chap-Kolm-2}. To make this heuristics precise, we need to analyze separately  the continuous and  the discrete components.

\begin{proof}[Proof of Theorem \ref{Prop-inv}]
We first consider  the continuous component. From \eqref{p-c} we see that the density of
$$\frac{\Gamma(\A+\C)\Gamma (\B+\C) \Gamma (\A+\B+2 t)}{\Gamma (\B+t)^2} \mathfrak p_t\topp c(\d x\mid \A,\B,\C)$$ is
\begin{multline*}
   \frac{|\Gamma(\A+t+\i \sqrt{x},\C-t+\i \sqrt{x} )|^2}{4 \pi   \sqrt{x}|\Gamma(2 \i \sqrt{x}|^2)} \cdot \frac{|\Gamma(\B+t+\i  \sqrt{x} )|^2}{\Gamma (\B+t)^2 }
 \\ =   \frac{|\Gamma(\A+t+\i \sqrt{x},\C-t+\i \sqrt{x} )|^2}{4 \pi   \sqrt{x}|\Gamma(2 \i \sqrt{x}|^2)} \cdot \prod_{k=0}^\infty \frac{1}{1+\frac{x}{(\B+t+k)^2}}\nearrow  \frac{|\Gamma(\A+t+\i \sqrt{x},\C-t+\i \sqrt{x} )|^2}{4 \pi   \sqrt{x}|\Gamma(2 \i \sqrt{x}|^2)} \mbox{ as $\B\to\infty$ }.
\end{multline*}
(Here we used $|\Gamma(x)/\Gamma(x+iy)|^2=\prod_k (1+y^2/(x+k)^2)$, see \cite[5.8.3]{NIST2010}.) The monotone convergence holds, because for $x>0$ we have
\begin{multline*}
   0<\sum_{k=0}^\infty \log\left(1+\frac{x}{(\B+t+k)^2}\right)<\sum_{k=0}^\infty \frac{x}{(\B+t+k)^2}<
\frac{x}{(\B+t)^2}+ \sum_{k=1}^\infty \frac{x}{(\B+t+k)(\B+t+k-1)} \\=\frac{x}{(\B+t)^2}+\frac{x}{\B+t} \to 0  \mbox{ as $\B\to\infty$ }.
\end{multline*}

 Next we consider the discrete components. We note that  the locations of atoms do not depend on parameter $\B$.
We compute the limit of masses of the atoms. For atoms in \eqref{pd-m}, as $\B\to\infty$ we have
\begin{multline*}
  \frac{\Gamma(\A+\C)\Gamma (\B+\C) \Gamma (\A+\B+2 t)}{\Gamma (\B+t)^2}m_k(t)
  \\=     \frac{ \Gamma(\A+\C)\Gamma (-\A+\C-2 t)}{\Gamma (-2 (\A+t)) }\cdot\frac{(\A+k+t) (\A+\C)_k (2 (\A+t))_k }{k! (\A+t) (\A-\C+2 t+1)_k}(-1)^k \frac{\Gamma (\B-\A)\Gamma (\A+\B+2 t)}{\Gamma (\B+t)^2}\cdot \frac{(\A+\B+2 t)_k}{  (\A-\B+1)_k}
\\ \to  \frac{ \Gamma(\A+\C)\Gamma (-\A+\C-2 t)}{\Gamma (-2 (\A+t)) }\cdot\frac{(\A+k+t) (\A+\C)_k (2 (\A+t))_k }{k! (\A+t) (\A-\C+2 t+1)_k}.%
\end{multline*}
This gives \eqref{their-atoms}.
For atoms in \eqref{pd-M}, as $\B\to\infty$ we have
\begin{multline*}
  \frac{\Gamma(\A+\C)\Gamma (\B+\C) \Gamma (\A+\B+2 t)}{\Gamma (\B+t)^2}M_k(t)
  \\= \frac{(\C+k-t)}{k! (\C-t)}\cdot\frac{\Gamma (\A-\C+2 t,\A+\C) }{\Gamma (2 (t-\C)) }
      \cdot
     \frac{ (\A+\C ,2 (\C-t))_k}{ (-\A+\C-2 t+1)_k }(-1)^k\frac{\Gamma (\B-\C+2 t)\Gamma (\B+\C)}{\Gamma (\B+t)^2 }\cdot \frac{(\B+\C)_k}{(-\B+\C-2 t+1)_k}
     \\ \to \frac{(\C+k-t)}{k! (\C-t)}\cdot\frac{\Gamma (\A-\C+2 t,\A+\C) }{\Gamma (2 (t-\C)) }
      \cdot
     \frac{ (\A+\C ,2 (\C-t))_k}{ (-\A+\C-2 t+1)_k }.%
 \end{multline*}
This gives \eqref{their-other-atoms}.

We can now prove \eqref{ps-invar}.
Fix a  Borel set $U$ and $s<t$. We will use Theorem \ref{Thm:Chapman}(ii), i.e., formula \eqref{Chap-Kolm-2}. Since the convergence of the densities is monotone, the locations of atoms do not depend on $\B$, and there are only finitely many atoms,
by the monotone convergence theorem for the continuous component, and  by convergence of the finite sums for the discrete component, we have
\begin{multline*}
   \int_\r \mathfrak p_{s}(\d x)\mathfrak p_{s,t}(x,U) =\int_\r \mathfrak p_{s}\topp c(\d x)\mathfrak p_{s,t}(x,U)+\int_\r \mathfrak p_{s}\topp d(\d x)\mathfrak p_{s,t}(x,U)
   \\=
   \int_\r   \lim_{\B\to\infty}  \mathfrak p_s\topp c(\d x\mid \A,\B,\C)\mathfrak p_{s,t}(x,U)+\int_\r    \lim_{\B\to\infty}\mathfrak p_s\topp d(\d x\mid \A,\B,\C)\mathfrak p_{s,t}(x,U)
   \\ = \lim_{\B\to \infty} \left( \int_\r   \mathfrak p_t\topp c(\d x\mid \A,\B,\C)\mathfrak p_{s,t}(x,U)+\int_\r  \mathfrak p_t\topp d(\d x\mid \A,\B,\C)\mathfrak p_{s,t}(x,U)\right)
   \\= \lim_{\B\to \infty} \int_\r   \mathfrak p_s(\d x\mid \A,\B,\C)\mathfrak p_{s,t}(x,U) \stackrel{ \eqref{Chap-Kolm-2}}{=}\lim_{\B\to \infty}  \mathfrak p_t(U\mid \A,\B,\C)=
    \mathfrak p_{t}(U).
\end{multline*}
\end{proof}
Formulas \eqref{Chap-Kolm-1} and \eqref{ps-invar} extend the formulas in \cite[Lemma 7.11]{CorwinKnizel2021} to a larger time domain and hence extend  the continuous dual Hahn  process $(\TT_s) $ to $s\in\r$.
\begin{definition}\label{Def 2}
  A continuous dual Hahn process $(\TT_s)_{-\infty<s<\infty}$ with real parameters $\A,\C$ such that $\A+\C>0$ is
  a family of  Markov    transition probabilities \eqref{P_st-3par} with  the family of $\sigma$-finite entrance laws \eqref{their-nu}.
\end{definition}
Since $(\TT_s)_{-\infty<s<\infty}$ is not a "stochastic process" in the usual probabilistic sense, we remark that
for any real $\tau$ and a measurable function $\varphi>0$ such that
$$
\mathfrak C:=\int_\r \varphi(x) \mathfrak p_\tau(\d x)<\infty,
$$
the family of entrance laws \eqref{their-nu} can be used to construct a Markov process $(X_t)$ on $(-\infty,\tau]$ with   probability measures as the marginal laws.
Versions of such constructions are well known; we follow \cite[Section 10]{dynkin1978sufficient}.
It is easy to see that for $t_0<t_1<\dots<t_n<t_{n+1}=\tau$, the consistent family of the joint  probability distributions for $(X_{t_0},X_{t_1},\dots, X_{t_n}, X_{\tau}) \in \r^{n+2}$ is
$$
\frac{1}{\mathfrak C}\mathfrak p_{t_0}(\d x_0) \mathfrak p_{t_0,t_1}(x_0,\d x_1)\dots \mathfrak p_{t_{n-1},t_n}(x_{n-1},\d x_n) p_{t_{n},\tau}(x_{n},\d x_{n+1})\varphi(x_{n+1}).
$$
Thus Kolmogorov's extension theorem  guarantees existence, and    Markov property holds
  with the initial law
$$
P(X_{t_0} = \d x)=\frac{1}{\mathfrak C} h_{t_0}(x)\mathfrak p_{t_0}(\d x)
$$
and with transition probabilities
$$
P(X_t=\d y|X_s=x)=\frac{1}{h_s(x)} \mathfrak p_{s,t}(x,\d y) h_t(y), \; t_0\leq s<t\leq \tau, x\in\r,
$$
where  for $t<\tau$ we define %
$$
h_t(x):=\int_\r \varphi(y)  \mathfrak p_{t,\tau}(x,\d y)
$$
with $h_\tau(x):=\varphi(x)$.
Note that $\varphi>0$ implies $h_t(x)>0$ for all $x\in\r$; in particular, if $x\not\in E_s$, then $h_t(x)=\varphi(-(\C-t)^2)$.

\section{Characterization by the conditional means and variances}\label{Sec:QH}
For a process $(X_t)_{t\geq 0}$ and $s<u$  consider the two-sided sigma fields
$\mathcal{F}_{ s, u}$   generated by $\{X_r: r\in (0,s]\cup[u,\infty)\}$. %
We will also use the past $\sigma$-fields  $\mathcal{F}_{ s}$   generated by $\{X_r: r\in (0,s]\}$.

The following is a characteristic property of the Markov process introduced in Definition \ref{Def-T}. %

\begin{theorem}
  \label{T-char}
  Let $(X_t)_{t\geq 0}$ be a square-integrable process   such that
  for all $t,s\geq 0$,
\begin{equation}\label{EQ: cov}
\E[X_t]=0,\: \E[X_sX_t]=\min\{t,s\}.
\end{equation}
Assume that for $t>0$, random variable $X_t$ has infinite support.
Then the following conditions are equivalent:
\begin{enumerate}[(i)]
  \item  For all $0\leq s<t<u$,
\begin{equation}
\label{EQ: LR} \E[{X_t}|{\mathcal{F}_{ s, u}}]=\frac{u-t}{u-s}
X_s+\frac{t-s}{u-s} X_u,
\end{equation}
 and there exist $\eta>0,\theta>-2$ such that
  \begin{equation}\label{EQ: q-Var}
\V [X_t|\mathcal{F}_{s,u }]
= \frac{(u-t)(t-s)}{1+u- s}\left( 1+\eta \frac{uX_s-sX_u}{u-s} +\theta\frac{X_u-X_s}{u-s}
+ \frac{(X_u-X_s)^2}{(u-s)^2}
 \right).
\end{equation}
\item
The law of $(X_{t})_{t\geq 0}$ is the same as the law of the process
\begin{equation}\label{T2X}
   \frac{\T_{t/2-(\A+\B)/2}}{\sqrt{(\A+\C)(\B+\C)}}+\frac{t^2-2(\A+\B+2 \C) t+(\A-\B)^2}{4\sqrt{(\A+\C)(\B+\C)}}  ,
\end{equation}
where $(\T_t)$  is a Markov process from Definition \ref{Def-T} with parameters $\A,\B,\C$ such that $\C\in\r$ and
 \begin{equation}
   \label{eta-theta2ABC}
   \A+\C=\frac{\theta-\sqrt{\theta^2-4}}{2\eta}, \quad \B+\C=\frac{\theta+\sqrt{\theta^2-4}}{2\eta}.
 \end{equation}
\end{enumerate}

\end{theorem}
We note that the expression in \eqref{eta-theta2ABC} should be interpreted as:
$$\frac{\theta\pm\sqrt{\theta^2-4}}{2\eta}=\begin{cases}
  \frac{\theta\pm\sqrt{\theta^2-4}}{2\eta}, & \theta\geq 2, \\ \\
   \frac{\theta\pm\i \sqrt{4-\theta^2}}{2\eta}, &-2<\theta<2 .
\end{cases}$$
Conditions $\eta>0$, $\theta>-2$ ensure that Assumption \ref{A1} holds.
\arxiv{Formulas \eqref{eta-theta2ABC} are equivalent to
\begin{equation}
  \label{abc2eta} \eta=\frac{1}{\sqrt{(\A+\C)(\B+\C)}}, \quad \theta=\frac{2\C+\A+\B}{\sqrt{(\A+\C)(\B+\C)}},
\end{equation}
compare \cite[(3.5) and (3.6)]{Bryc:2009}.
In particular, if $\A=\alpha-\i \beta$ and $\B=\alpha+\i \beta$ with $\beta\neq 0$ as in Assumption \ref{A1}(ii), then
$\theta=2 (\C+\alpha)/|\C+\alpha+\i \beta|$, so $|\theta|<2$. On the other hand, under Assumption \ref{A1}(i), $2\C+\A+\B\geq 2(\A+\C)>0$ so \eqref{abc2eta} gives $\theta>0$.

}

In the terminology of \cite{Bryc-Matysiak-Wesolowski-04}, Theorem \ref{T-char}  says  that the following conditions are equivalent:
\begin{enumerate}[(i)]
  \item  $(X_t)_{t\geq 0}$  is a quadratic harness in standard form with parameters $\sigma=0$, $\tau=1$, $\eta>0$, $\theta>-2$ and infinite support for $X_t$, $t>0$.
      \item $(X_{t})_{t\geq 0}$ is a deterministic transformation \eqref{T2X}
of the Markov process $(\T_t)$ %
with parameters \eqref{eta-theta2ABC}.
\end{enumerate}

 Theorem \ref{T-char} also relates the finite dimensional distributions   of
 Markov processes $(\T_t)$ with the same values of sums $\A+\C$ and $\B+\C$.
\begin{corollary}
  If $(\T_s)_{s\geq -(\A+\B)/2}$ and $(\T_s')_{s\geq -(\A'+\B')/2}$ are two Markov processes from Definition \ref{Def-T} %
  with parameters
  $\A,\B,\C$ and $\A',\B',\C'$ respectively such that $\A+\C=\A'+\C'$ and $\B+\C=\B'+\C'$,
  then
 the laws of  $(\T_s)$ and $(\T_s')$ differ only by a   deterministic shift of time and a deterministic shift of space:
  $$
 \mathcal{L}\left(\tfrac{(\A'-\B')^2}{4} +\TT'_{s-(\A'+\B')/2}\right)_{s\geq 0}= \mathcal{L}\left(\tfrac{(\A-\B)^2}{4}+\TT_{s-(\A+\B)/{2}}\right)_{s\geq 0}.
 $$
\end{corollary}
\begin{remark}
  By inspecting formulas for the $\sigma$-finite entrance law and transition probabilities, we see that the finite-dimensional ($\sigma$-finite) joint distributions for the process $(\TT_{s+(\C-\A)/2})_{s\in\r}$ depend only on   $\A+\C$.
\end{remark}

 \subsection{ Proof of Theorem \ref{T-char}(ii)$\Rightarrow$(i)}
 For $-\A<s<t<u<\C$, this implication was verified  in \cite[Section 3]{Bryc:2009} by a direct computation with the conditional densities, and formula \eqref{T2X} comes from that paper.
 This is not a feasible approach for the time interval where the conditional laws may be of mixed type,
so we will rely on properties of polynomials, see \cite{Bryc-Matysiak-Wesolowski-04} and \cite[Section 3]{Bryc-Wesolowski-08}.
(Even with this technique, the proof still involves several long calculations that we will omit.)

From
explicit expressions for polynomials $p_1(x;t)$ and    $p_2(x;t)$, see \eqref{p-mart},
we read out the first two moments:
\begin{equation}
  \label{ETVarT}
  \E[\T_t]=\B \C + \A \B + \A\C + 2 \C t -t^2, \Var[\T_t]=(\A + \C) (\B + \C) (\A + \B + 2 t).
\end{equation}
Thus
\begin{equation}\label{ETVarT-shift}
\E[\T_{t-(\A+\B)/2}]=(\A+\B+2\C)t-t^2-(\A-\B)^2/4, \Var[\T_{t-(\A+\B)/2}]=2 t (\A+\C)(\B+\C).
\end{equation}
This shows that the first two moments of $X_{t}$ and of random variable \eqref{T2X} are the same, as claimed.

Since $Q_1(y;x,t,s)=2 \C (s-t)-s^2+t^2-x+y$ we see that
$\T_{t} - \E(\T_{t})$ is a martingale. Thus   $Cov[\T_s,\T_t]=\Var[\T_{\min\{ s,t\}}]$, and the covariance of process \eqref{T2X}   matches the covariance   in \eqref{EQ: cov}.

It remains to confirm that the first two conditional moments match. We begin by  removing the  quadratic  component from the mean, so we will be working with process
\begin{equation}
  \label{T2Y}
  Y_t:=\T_t+t^2.
\end{equation}
We have
\begin{equation}\label{LR+QV}
\E [Y_t]=\B \C + \A \B + \A\C + 2 \C t, \mbox{ and }
 \Var[Y_t]=(\A + \C) (\B + \C) (\A + \B + 2 t) >0.
\end{equation}
\arxiv{Note that formula \eqref{eta-theta2ABC} implies that $(\A+\C)(\B+\C)=1/\eta^2>0$.}

Our goal is to show that $(Y_t)$ is a (general) quadratic harness, i.e., that it has conditional mean
\begin{equation}\label{Y-cond-LR}
  \E[Y_t|Y_s,Y_u]=\frac{u-t}{u-s}Y_s+\frac{t-s}{u-s}Y_u, \; s<t<u,
\end{equation}
and conditional variance
\begin{equation}\label{Y-cond-QV}
  \Var[Y_t|Y_s,Y_u]=\frac{(u-t)(t-s)}{1+2(u-s)}\left(4\frac{uY_s-sY_u}{u-s}+\frac{(Y_u-Y_s)^2}{(u-s)^2}\right), \; s<t<u.
\end{equation}
Since $(Y_t)$ is a Markov process, the formulas for conditional moments \eqref{Y-cond-LR} and \eqref{Y-cond-QV} are of course the same if we condition with respect to the two-tail $\sigma$-field $\mathcal{F}_{s,u }$ generated by $(Y_t)$.
Once \eqref{Y-cond-LR} and \eqref{Y-cond-QV} are established, routine but cumbersome calculations based on moments \eqref{LR+QV} then verify that the deterministic transformation
\begin{equation}
  \label{Y2X}
  X_{t}=\frac{Y_{t/2-\tau}-\E(Y_{t/2-\tau})}{\sqrt{\Var(Y_{1/2-\tau})}}=
  \frac{Y_{t/2-\tau}-(\A\B+\C t)}{\sqrt{ (\A + \C) (\B + \C) }}
\end{equation}
with $\tau=(\A+\B)/2$, converts process $(Y_t)$ into a quadratic harness $(X_t)$ in standard form:  the  moments of $X_t$ become \eqref{EQ: cov}, formula \eqref{EQ: LR} is a direct consequence of \eqref{Y-cond-LR},
and a longer calculation verifies that \eqref{EQ: q-Var} follows from \eqref{Y-cond-QV}.
For the latter, we apply  \eqref{Y2X} to the left hand side of \eqref{EQ: q-Var}. From
expression
$$4\frac{uY_s-sY_u}{u-s}+\frac{(Y_u-Y_s)^2}{(u-s)^2}$$
on the right hand side of \eqref{Y-cond-QV}, we get
\begin{equation*}
   4\frac{(u/2-\tau)Y_{s/2-\tau}-(s/2-\tau)Y_{u/2-\tau}}{u/2-s/2}
  +4\frac{(Y_{u/2-\tau}-Y_{s/2-\tau})^2}{(u-s)^2}.
\end{equation*} Applying the inverse of transformation \eqref{Y2X}, after a tedious but elementary calculation, we arrive at  \eqref{EQ: q-Var}, up to a multiplicative constant.
 (Transformation \eqref{Y2X}  is just a different way of writing \eqref{T2X}. One can also apply formulas in \cite[Proposition 1.1]{Bryc:2009}.) We omit the details of this calculation.

To verify \eqref{Y-cond-LR} and \eqref{Y-cond-QV}, we
 use the orthogonal martingale polynomials for the process $(Y_t)$. The three step recursion for these polynomials is recalculated from
 the recursion for polynomials \eqref{p-mart}.
Recall that if monic orthogonal polynomials $\{p_n\}$ for (the law of) a random variable $T$ satisfy recursion
$$
xp_n(x)=p_{n+1}(x)+\beta_n p_n(x)+\gamma_n p_{n-1}(x),
$$
then the monic orthogonal polynomials $\{q_n\}$ for (the law of)  $Y=(T-\mu)/\sigma$
satisfy recursion
$$
x q_n(x)=q_{n+1}(x)+\frac{\beta_n-\mu}{\sigma}q_{n}(x)+\frac{\gamma_n}{\sigma^2}q_{n-1}(x).
$$

After a calculation, we verify that monic orthogonal martingale polynomials $q_n(x;t)$ for the process $(Y_t)_{t\geq -\tau}$ satisfy recursion
\begin{equation}
  \label{X-recursion}
  x  q_n(x;t)=q_{n+1}(x;t)+b_n(t)q_{n}(x;t)+c_n(t)q_{n-1}(x;t),\quad n\geq 0,
\end{equation}
with $q_{-1}(x;t)=0$, $q_0(x;t)=1$, where the coefficients
\begin{equation}\label{bc}
  b_n(t)=\alpha_n+\beta_n t, \quad c_n(t)=\gamma_n+\delta_n t
\end{equation} are given by
\begin{eqnarray}
    \alpha_n&=& \A \B+\A\C+\B\C+(2(\A+\B+\C)-1 )n    \label{alpha},\\
  \beta_n&=&2 (\C + n) \label{beta},
  \\
\gamma_n&=&n (\A+\B+n-1) (\A+\C+n-1) (\B+\C+n-1)  \label{gamma},\\
 \delta_n&=&2 n (\A+\C+n-1) (\B+\C+n-1) \label{delta}.
\end{eqnarray}

To determine the conditional moments we use the following  criterion.
\begin{lemma}
  \label{Lem:p2m}
  Fix $k=1,2,\dots$. Suppose that $\varphi(x,y)$ is a polynomial such that for all $n=0,1,\dots$ we have
  \begin{equation}
    \label{War-q}
    \E[ Y_t^k q_n(Y_u;u)|Y_s]=\E[\varphi(Y_s,Y_u) q_n(Y_u;u)|Y_s].
  \end{equation}
  Then
  $$\E[Y_t^k|Y_s,Y_u]=\varphi(Y_s,Y_u).$$
\end{lemma}
\begin{proof}
By Dynkin's $\pi-\la$ lemma, it is enough to show that for any pair of bounded measurable functions $f,g:\r\to[0,\infty)$ we have
\begin{equation}
  \label{condExp}
  \E[ f(Y_s) Y_t^k g(Y_u)]=\E[ f(Y_s) g(Y_u)\varphi(Y_s,Y_u)].
\end{equation}
By our assumption, \eqref{condExp} holds  if $g(Y_u)$ is replaced by a polynomial $p(Y_u)$, a linear combination of the polynomials $\{q_n(Y_u;u): n=0,1,\dots\}$.

By Lemma \ref{Lem:unique},   the law of  $Y_u$ is determined by moments.  Recall  that for a probability measure determined by moments,
 polynomials are dense in $L_2$,  \cite[Corollary 2.3.3]{Akhiezer}.
So if  $g$ is an arbitrary bounded measurable function, then
   for any $\eps>0$, there exists a polynomial $p$ such that
 $\E[|g(Y_u)-p(Y_u)|^2]<\eps^2$. Since \eqref{condExp}  holds for  the polynomial $p$, the difference between  the left hand side and the right hand side of \eqref{condExp} is  arbitrarily small. Indeed,
  \begin{multline*}
     \left|\E[ f(Y_s) Y_t^k g(Y_u)]- \E[ f(Y_s) g(Y_u)\varphi(Y_s,Y_u)]\right|
     \\ \leq
      |\E[ f(Y_s) Y_t^k (g(Y_u)- p(Y_u))]|+ |\E[ f(Y_s) \left(g(Y_u)-p(Y_u)\right)\varphi(Y_s,Y_u)]|.
  \end{multline*}
  By the Cauchy-Schwarz inequality, the first term can be bounded by
$$
   | (\E[f^2(Y_s) Y_t^{2k}])^{1/2}(\E[|g(Y_u)-p(Y_u)|^2])^{1/2}
    \leq (\E[f^2(Y_s) Y_t^{2k}])^{1/2}\eps,
    $$
 and similarly the second term is at most
 $$(\E[f^2(Y_s)\varphi^2(Y_s,Y_u)])^{1/2}(E[|g(Y_u)-p(Y_u)|^2])^{1/2}
  \leq (\E[f^2(Y_s)\varphi^2(Y_s,Y_u)])^{1/2}\eps.
 $$ %
Since $\eps>0$ is arbitrary,  \eqref{condExp} follows.
\end{proof}

For the subsequent calculations, we write recursion \eqref{X-recursion} in the  vector form as
\begin{equation}
  \label{vect-J}
  x  \vec{q}(x;t)= \vec q (x;t)\mJ(t),
\end{equation}where $\vec q(x;t)=[q_0(x;t),q_1(x;t),\dots]$ and $\mJ(t)$ is the Jacobi matrix
$$
\mJ(t)=
\begin{bmatrix}
 b_0(t) & c_1(t) & 0 & 0  & \dots & 0 & \dots
   \\
 1 & b_1(t) & c_2(t) & 0 &  \dots &0
   &  \dots \\
 0 & 1 & b_2(t) & c_3(t) & \ddots &\vdots
   &  \\
 0 & 0 & 1 & b_3(t) & \ddots & 0
   &  \dots \\
\vdots & \vdots & \ddots & \ddots & \ddots & c_n(t)
   & \ddots  \\
 0 & 0 &  & 0 & 1 & b_n(t) &\ddots &
    \\
 \vdots & \vdots &   &   &   \ddots & \ddots &\ddots
\end{bmatrix}
$$
with the diagonal entries given by \eqref{bc}. From \eqref{bc} we see that the Jacobi matrix depends linearly on $t$,
\begin{equation}
  \label{J} \mJ(t)=\mY+t\mX.
\end{equation}
Ref. \cite{Bryc-Matysiak-Wesolowski-04} indicates that linearity of regression \eqref{Y-cond-LR} (the so called harness property) is related to   the fact that $\mJ$ depends linearly  on $t$,
 and that quadratic conditional variance \eqref{Y-cond-QV}  is related to a quadratic relation between the matrices $\mX, \mY$.
 However,   process $(Y_t)$ here  has a different covariance, and the emphasis in \cite{Bryc-Matysiak-Wesolowski-04} was on  the  converse implication, so we shall work out the formulas that are pertinent to our case anew.

\begin{proof}[Proof of \eqref{Y-cond-LR}] We use Lemma \ref{Lem:p2m} with $k=1$ and $\varphi(x,y)=\frac{u-t}{u-s}x+\frac{t-s}{u-s}y$. To verify assumption \eqref{War-q}, we write it in vector form as
 \begin{equation}
   \label{LR-q}
   \E [Y_t \vec q (Y_u;u)|Y_s]=\E \left[\left(\tfrac{u-t}{u-s}Y_s+\tfrac{t-s}{u-s}Y_u\right) \vec q (Y_u;u)\middle|Y_s\right].
 \end{equation}
 To verify \eqref{LR-q}, we use the vector form of the martingale property, component-wise.
That is, we write martingale identity $$
\E [Y_t q_n(Y_u;u)|Y_s]=\E \left[Y_t \E[q_n(Y_u;u)|Y_t]\middle |Y_s\right] =\E [Y_t q_n(Y_t;t)|Y_s]
$$
  in the vector form, and combine it with the vector form \eqref{vect-J}  of the three step recursion. We get
  \begin{equation}
  \E [Y_t \vec q (Y_u;u)|Y_s]=\E [Y_t \vec q(Y_t;t)|Y_s] =\E[\vec q(Y_t;t)\mJ(t)|Y_s]= \E[\vec q(Y_t;t)|Y_s]\mJ(t)= \vec q(Y_s;s)\mJ(t).
\end{equation}
Similarly, we have
 $$\E [Y_u \vec q (Y_u;u)|Y_s]=\vec q(Y_s;s)\mJ(u).$$
Since $Y_s\vec q(Y_s;s)=\vec q(Y_s;s)\mJ(s)$,
we see that formula \eqref{LR-q},   is a consequence of a simple algebraic identity
\begin{equation}
  \label{LR-J} \mJ(t)=\frac{u-t}{u-s}\mJ(s)+\frac{t-s}{u-s}\mJ(u),
\end{equation}
left-multiplied by $\vec q(Y_s;s)$. Identity  \eqref{LR-J} holds as $\mJ(t)=t\mX+\mY$ is linear in variable $t$.
Hence, by Lemma \ref{Lem:p2m}
formula \eqref{Y-cond-LR} holds.
\end{proof}
\begin{proof}[Proof of \eqref{Y-cond-QV}]
   This proof is based on a similar plan: we shall deduce \eqref{Y-cond-QV}  from  Lemma \ref{Lem:p2m} with $k=2$  using an algebraic identity
   \begin{equation}
  \label{q-comm}
  \mX \mY-\mY\mX= \frac12\mX^2+ 2 \mY
\end{equation}
for the two components of the Jacobi matrix \eqref{J},
compare \cite[formula (1.1)]{Bryc-Matysiak-Wesolowski-04}. The arguments rely on several cumbersome calculations which we shall only indicate. (We used a computer algebra system to complete several calculations, with  \eqref{q-comm} verified by representing infinite matrices $\mX, \mY$ as  \eqref{D2X} and \eqref{D2Y}.)%

The first step is   to  rewrite formula \eqref{Y-cond-QV} in expanded form. A calculation shows that \eqref{Y-cond-QV} is equivalent to the following
\begin{multline}\label{QV-sub}
\E[Y_t^2|Y_s,Y_u] = \frac{(1+2 u-2 t)(u-t) }{(1+2 u-2 s) (u-s)}Y_s^2+\frac{(1+2 t-2 s) (t-s)}{(1+2 u-2 s) (u-s)}Y_u^2\\+\frac{4 (t-s) (u-t)}{(1+2 u-2 s) (u-s)} Y_sY_u
+\frac{4 u (t-s) (u-t)}{(1+2 u-2 s) (u-s)}Y_s-\frac{4 s (t-s) (u-t)}{(1+2 u-2 s) (u-s)} Y_u.
\end{multline}
So   assumption   \eqref{War-q} in  vector form is:
\begin{multline}\label{QV-sub+}
\E[Y_t^2\vec q(Y_u;u)|Y_s] = \frac{(1+2 u-2 t)(u-t) }{(1+2 u-2 s) (u-s)}Y_s^2\vec q(Y_s;s)
\\+\frac{(1+2 t-2 s) (t-s)}{(1+2 u-2 s) (u-s)}
\E[Y_u^2\vec q(Y_u;u)|Y_s]+\frac{4 (t-s) (u-t)}{(1+2 u-2 s) (u-s)} Y_s\E[Y_u\vec q(Y_u;u)|Y_s]\\
+\frac{4 u (t-s) (u-t)}{(1+2 u-2 s) (u-s)}Y_s\vec q(Y_s;s)-\frac{4 s (t-s) (u-t)}{(1+2 u-2 s) (u-s)} \E[Y_u\vec q(Y_u;u)|Y_s].
\end{multline}
Next, we note that by the martingale property for the sequence of polynomials $\vec q(x;t)$, we have
$$\E[Y_t^2 \vec q(Y_u;u)|Y_s] = \E\left[Y_t^2 \E[\vec q(Y_u;u)|Y_t]\middle|Y_s\right] = \E[Y_t^2  \vec q(Y_t;t)|Y_s]=  \vec q(Y_s;s)\mJ^2(t).
$$
Similar calculations apply to each of the terms  on the right hand side of \eqref{QV-sub+}.
So to deduce \eqref{QV-sub}, and hence \eqref{Y-cond-QV} from Lemma \ref{Lem:p2m}, it is enough to show that the Jacobi matrices satisfy identity
\begin{multline*}
   \mJ^2(t) =
    \frac{ (1+2u-2 t)(u-t)}{(1+2 u-2 s) (u-s)}\mJ^2(s)+\frac{(1+2 t-2 s) (t-s)}{(1+2 u-2 s) (u-s)}\mJ^2(u)+\frac{4 (t-s) (u-t)}{(1+2 u-2 s) (u-s)} \mJ(s) \mJ(u)
\\+\frac{4 u (t-s) (u-t)}{(1+2 u-2 s) (u-s)}\mJ(s)-\frac{4 s (t-s) (u-t)}{(1+2 u-2 s) (u-s)} \mJ(u)
\end{multline*}
for all $s<t<u$.
After substituting \eqref{J} into this expression, a  lengthy calculation shows that for $s<t<u$, the identity is equivalent to \eqref{q-comm}. (For a similar result of this type, see \cite[Proposition 4.9]{Bryc-Matysiak-Wesolowski-04}.)

The final step is to prove that \eqref{q-comm} holds. Here, we use the explicit form of the matrices $\mX,\mY$.
Matrix $\mX$ is bi-diagonal, with the sequence $(\beta_0,\beta_1,\dots)$ on the main diagonal, and $(\delta_1,\delta_2,\dots)$ above the main diagonal.
Matrix $\mY$ is tri-diagonal with $1$'s below the main diagonal, $(\alpha_0,\alpha_1,\dots)$ on the main diagonal, and $(\gamma_1,\gamma_2,\dots)$ above the main diagonal.
Relation \eqref{q-comm} becomes a system of recursions for these coefficients. Using the explicit formulas (\ref{alpha}-\ref{delta}),  another lengthy calculation   confirms that \eqref{q-comm} indeed holds.
This ends the proof of \eqref{Y-cond-QV}.
\end{proof}
\begin{remark}\label{Rem5.5}
   A somewhat more conceptual approach  to  \eqref{q-comm} is described  in \cite[Section 4.4]{Bryc-Matysiak-Wesolowski-04}. In this approach, one represents Jacobi matrix as  a matrix of an operator on polynomials in variable $z$ in the basis of monomials. The three step recursion \eqref{X-recursion} is then encoded by $\mJ_t=t\mX+\mY$ with
 \begin{equation}
   \label{D2X}
    \mX=2(\C+z\partial_z)+2(\A+\C+z\partial_z)(\B+\C+z\partial_z)\partial z,
 \end{equation}
    \begin{equation}
   \label{D2Y}
   \mY=z+(\A \B+\A\C+\B\C)+(2(\A+\B+\C)-1)z\partial_z+(\A+\B+z\partial_z)(\A+\C+z\partial z)(\B+\C+z\partial z)\partial_z,
 \end{equation}
where $z^n$ represents the $n$-th orthogonal polynomial.
   In principle, verification of \eqref{q-comm} in this form is just  a long calculation which uses  the product rule
   $\partial_z (z f(z))=f(z)+z f'(z)$ to swap the order of operators $\partial_z z= 1+z\partial_z$ multiple times. The plan here is to rewrite all mixed products of operators on both sides of the identity \eqref{q-comm} in normal (Wick) order, i.e., to express both sides as  linear combinations of the operators $z^m \partial_z^k$, $m,k\geq0$, and then compare the coefficients. Instead, we used a computer algebra system to verify \eqref{q-comm} by this technique.
\arxiv{
 \subsection*{Mathematica code for Remark \ref{Rem5.5}}\label{MathCd}
A reader who have attempted a direct verification of \eqref{q-comm} as described in our sketch of proof might appreciate  an implementation in a symbolic computer algebra system. (Some longer outputs are suppressed.)

Define the operators:

\noindent\(\pmb{\text{XX}[\text{z$\_$}]=2(C (\#)+z D[\#,z])+2 (A+C)(B+C)D[\#,z]+2(A+B+2C) z D[\#,\{z,2\}] +2z D[z D[\#,\{z,2\}],z]\&}\)

\begin{multline*}
   \pmb{\text{YY}[\text{z$\_$}]=(z+A B+A C+B C)\#+(2(A+B+C)-1) z D[\#,z]}+\\ \pmb{2 z D[z D[\#,z],z] +z D[z D[z D[\#,\{z,2\}],z],z]+ (A+B)(A+C)(B+C)D[\#,z] +}\\\pmb{((A+B)(A+C)+(A+B)(B+C)+(A+C)(B+C))z D[\#,\{z,2\}]+} \\
\pmb{((A+B)+(A+C)+(B+C))z D[z D[\#,\{z,2\}],z]\&}
\end{multline*}

Confirm  match with \eqref{J} and \eqref{X-recursion}, where  $\A=A$, $\B=B$, $\C=C$:

\begin{doublespace}
\noindent\(\pmb{\text{Collect}[\text{YY}[z][ z{}^{\wedge}n]+ t*\text{XX}[z][ z{}^{\wedge}n],\{t,z\},\text{FullSimplify}]}\)
\end{doublespace}

\noindent\(n (-1+A+B+n) (-1+A+C+n) (-1+B+C+n) z^{-1+n}+(-n+2 n (C+n)+B (C+2 n)+A (B+C+2 n)) z^n+z^{1+n}+t \left(2 n (-1+A+C+n) (-1+B+C+n) z^{-1+n}+2
(C+n) z^n\right)\)

\medskip
Define the left hand side LHS   and the right hand side RHS of \eqref{q-comm}, acting on $f(z)$:

\begin{doublespace}
\noindent\(\pmb{\text{LHS}=\text{XX}[z][\text{YY}[z][f[z]]]-\text{YY}[z][\text{XX}[z][f[z]]]\text{//}\text{FullSimplify}}\\
\pmb{\text{RHS}=1/2 \text{XX}[z][\text{XX}[z][f[z]]]+2 \text{YY}[z][f[z]]\text{//}\text{FullSimplify}}\)
\end{doublespace}

Prove \eqref{q-comm}:

\begin{doublespace}
\noindent\(\pmb{\text{LHS}-\text{RHS}\text{//}\text{FullSimplify}}\)
\end{doublespace}

\begin{doublespace}
\noindent\(0\)
\end{doublespace}
}
\end{remark}
 \subsection{ Proof of Theorem \ref{T-char}(i)$\Rightarrow$(ii)}
Process  $(X_t)$ is defined on $[0,\infty)$ and  satisfies (\ref{EQ: cov}-\ref{EQ: q-Var}), so    by \cite[Theorem 2.5]{Bryc-Matysiak-Wesolowski-04} it has finite moments of all order.
Then \cite[Theorem 4.1]{Bryc-Matysiak-Wesolowski-04} implies that process $(X_t)$ has orthogonal martingale polynomials $\{p_n(x;t)\}$, and their three step recursion is determined uniquely by the coefficients in
 \eqref{EQ: cov}, \eqref{EQ: LR} and \eqref{EQ: q-Var}. However, by the first part of the theorem, the same holds for the process $\widetilde X$ obtained by transformation \eqref{T2X} of the
 continuous dual Hahn process. So both processes have the same orthogonal martingale polynomials. It remains to show that processes $X$ and $\widetilde X$ have the same finite dimensional distributions.

For any $t>0$, both processes have the same orthogonal polynomials and hence the same moments
$$\E\left[X_{t}^{n}\right]=\E\left[\widetilde X_{t}^{n}\right].$$
In view of Lemma \ref{Lem:unique}, this means that both processes have the same univariate laws.
In addition, by martingale property of the orthogonal martingale polynomials, for every $n$, one can find a polynomial $\varphi_n$ such that
$$\E\left[X_{t}^{n}| \mathcal{F}_s\right]=\varphi_n(X_s) \mbox { and } \E\left[\widetilde X_{t}^{n}\middle| \widetilde{\mathcal{F}}_s\right]=\varphi_n(\widetilde X_s) .$$
(Here $\mathcal{F}_s$ and $\widetilde{\mathcal{F}}_s$ are the past $\sigma$-fields.)
By Lemma \ref{Lem:unique}, this proves that conditional laws are the same, at every point of the support of $X_s$. Since  $\widetilde X$  is a Markov process, process $X$ is also Markov, with the same transition probabilities. So the finite dimensional distributions for both processes are the same.
 \qed

\subsection*{Acknowledgement} The author thanks Wojciech Matysiak for helpful comments that improved an early draft of this paper and for Remark \ref{Rem:Wojtek},  and to Jacek Weso{\l}owski for the discussions that lead to the Markov processes described below Definition \ref{Def 2}.
This research was partially supported by Simons Foundation/SFARI Award Number: 703475.

\end{document}